\documentclass[12pt,reqno]{amsart}
\usepackage{amscd,amssymb,amsthm}
\usepackage{graphicx}
\usepackage{enumerate}
\theoremstyle{plain}
\newtheorem{theorem}{Theorem}[section]
\newtheorem{corollary}[theorem]{Corollary}
\newtheorem{lemma}[theorem]{Lemma}
\newtheorem{proposition}[theorem]{Proposition}
\newcommand{\floor}[1]{\left\lfloor{#1}\right\rfloor}

\newcommand{\ds}{\displaystyle}
\newcommand{\eps}{{\varepsilon}}

\newcommand{\round}{{\rm Round}}

\newcommand{\adh}[1]{\overline{#1}}
\newcommand{\bg}{\bigskip\goodbreak}

\newenvironment{enumeratea}{\begin{enumerate}%
	[\upshape (a)]}{\end{enumerate}}

\title[On Certain Sums Related to the Largest Odd Divisor]
{On Certain Sums Related to the Largest Odd Divisor}
\author{Omran Kouba}
\address{Department of Mathematics \\
Higher Institute for Applied Sciences and Technology\\
P.O. Box 31983, Damascus, Syria.}
\email{omran\_kouba@hiast.edu.sy}

\begin{document}
\date{\today}
\begin{abstract}
In this paper, we consider certain finite sums related to the ``largest odd divisor'', and  we obtain,
using simple ideas and recurrence relations, sharp upper and lower bounds for these sums.
\end{abstract}
\smallskip\goodbreak

\maketitle

\section{\bf Introduction }\label{sec1}
\parindent=0pt
\quad For a positive integer $k$, let $\alpha(k)$ be the largest odd divisor of $k$. 
So $\alpha$ is a very simple arithmetic function that can be defined using the recurrence relations :
\begin{equation}\label{E:eq1}
\alpha(2n-1)=2n-1,\quad\hbox{and}\quad \alpha(2n)=\alpha(n).
\end{equation}
In what follows we will study some properties related to several sums containing $\alpha$. In particular,  for a positive integer $n$, we will consider the following three sums :
\begin{align}
V(n)&=\sum_{k=1}^n\frac{\alpha(k)}{k},\label{E:eq2}\\
U(n)&=\sum_{k=1}^n\alpha(k),\label{E:eq3}\\
G(n)&=\sum_{k=1}^n\frac{n+1-k}{k}\alpha(k)=(n+1)V(n)-U(n).\label{E:eq4}
\end{align}
\bg
\quad Bounds for  $G(n)$  were proposed by Mih\'aly Bencze in \cite{ben}  and, as we will see in this paper,
the proposed bounds  there are not sharp. Also, questions concerning bounds for $V(n)$ and $U(n)$ can be found in several regional or
national Mathematical Olympiad problems, see \cite{put} and \cite{bay} for example.\bg
\quad Now, let us fix some notation. For a nonnegative integer $m$, we will denote by $I_m$ the set of integers $k$ satisfying
$2^m\leq k<2^{m+1}$. As usual, the logarithm in base $2$ will be denoted by $\lg$, and the floor function
will be denoted by $\floor{\cdot}$. Clearly we have following the equivalence $\floor{\lg k}=m\iff k\in I_m$.
\bg
\quad Also, if a nonnegative integer $n$ have the following binary representation
$$n=\sum_{k=0}^m\eps_k2^k,\quad\hbox{with $\eps_k\in\{0,1\}$ for every $k$,}$$
we write $n=(\eps_m\eps_{m-1}\cdots\eps_1\eps_0)_2$. We do not suppose that $\eps_m=1$ but clearly
we have $n\in I_m\iff \eps_m=1$. Finally,
if $\eps\in\{0,1\}$ we will write $\adh{\eps}$ to denote $1-\eps$.

\smallskip

\quad The paper is organized as follows. In section~\ref{sec2}, we gathered properties of $V$,
in particular we prove in Theorem \ref{t5} that
\begin{equation*}
\frac{2n^2+1}{3n}\leq V(n)\leq\frac{2n(n+2)}{3(n+1)}.
\end{equation*}

\quad In section~\ref{sec3}, we
find the properties of $U$, and particularly we find in Theorem \ref{t2} a precise
version of the following inequality
\begin{equation*}
\frac{n^2+2}{3}\leq U(n)\leq\frac{n^2+2n}{3}.
\end{equation*}

\quad In section~\ref{sec4}, the properties of $G$ are considered. We prove among other
results that
\begin{equation*}
\frac{n^2+2n}{3}-\theta_n\leq G(n)\leq\frac{n^2+2n}{3}
\end{equation*}
where
$\theta_n=\frac{1}{9}\left(\floor{\lg n}+\round(2^{\floor{\lg n}}/3)2^{-\floor{\lg n}}\right)$, where $\round(\cdot)$ is the nearest integer function. 

\quad Moreover,
we prove that all these inequalities are sharp in the sense that equality holds infinitely many times in the upper and also in the lower bounds. We also characterize, in each case, the values of $n$ where the equality sign holds. 

\quad Finally, we propose some problems that could be solved by the materials proposed in this article.

\bg
\section{\bf Properties of $V$ }\label{sec2}

\quad Our first result is about the recurrence relations satisfied by $V$, these relations are used to obtain sharp upper and lower bounds for $V$.
\bg

\begin{proposition}\label{p1} The function $V$ satisfies the following properties :

\begin{enumeratea}
\item  For each positive integer $n$, we have 
\begin{equation*}
V(2n)=n+\frac12 V(n)\quad\hbox{and}\quad V(2n+1)=n+1+\frac12 V(n).
\end{equation*}\label{itm11}
\item   For each positive integer $n$, we have
$\ds \frac{2n}{3}<V(n)<\frac{2n+2}{3}$.\label{itm12}
\end{enumeratea}
\end{proposition}

\begin{proof}
 Indeed, using the properties of $\alpha$, we can write
\begin{align*}
V(2n)&=\sum_{k=1}^{2n}\frac{\alpha(k)}{k}=
\sum_{k=1}^{n}\frac{\alpha(2k-1)}{2k-1}+\sum_{k=1}^{n}\frac{\alpha(2k)}{2k},\\
&=\sum_{k=1}^{n}1+\frac12\sum_{k=1}^n\frac{\alpha(k)}{k}=n+\frac12 V(n),\\
\end{align*}
and
$$V(2n+1)=V(2n)+\frac{\alpha(2n+1)}{2n+1}=V(2n)+1=n+1+\frac12 V(n).$$
So, we have proved the recurrence relations in (\ref{itm11}).\bg
\quad Now, we will prove by induction on $m$ the following property :
$$\mathcal{Q}_m:\qquad \forall\,n\in I_m,\quad \frac{2n}{3}<V(n)<\frac{2n+2}{3}.
$$
Since $V(1)=1$ we see immediately that $\mathcal{Q}_1$ is true. Let us suppose that
$\mathcal {Q}_m$ is true for some $m\geq1$, and consider $n\in I_{m+1}$. There are two cases :

\medskip
\begin{itemize}
\item $n=2p$ for some $p\in I_m$. Then
$\frac{2p}{3}<V(p)<\frac{2p+2}{3}$, and
$$p+\frac{p}{3}<p+\frac12 V(p)<p+\frac{p+1}{3},$$
and by (\ref{itm11}) this is equivalent to
$\frac{2n}{3}<V(n)<\frac{2n+1}{3}$.

\medskip
\item $n=2p+1$ for some $p\in I_m$.
Then $\frac{2p}{3}<V(p)<\frac{2p+2}{3}$, and
$$1+p+\frac{p}{3}<1+p+\frac12 V(p)<1+p+\frac{p+1}{3},$$
and, again by (\ref{itm11}) this is equivalent to
$\frac{2n+1}{3}<V(n)<\frac{2n+2}{3}$.
\end{itemize}

We conclude that $\frac{2n}{3}<V(n)<\frac{2n+2}{3}$ for every $n\in I_{m+1}$.  This
achieves the proof of the induction step : $\mathcal{Q}_m\Longrightarrow \mathcal {Q}_{m+1}$,
and completes the proof of Proposition \ref{p1}.
\end{proof}

\quad The following corollary follows from Theorem 1. in \cite{kou}, and the fact that
$$\lim_{n\to0}\frac{V(n)}{n}=\frac{2}{3}$$
\bg
\begin{corollary}
For every continuous function on $[0,1]$ we have
\begin{equation*}
\lim_{n\to\infty}\frac{1}{n}\sum_{k=1}^nf\left(\frac{k}{n}\right)\frac{\alpha(k)}{k}=\frac{2}{3}\int_0^1f(x)\,dx
\end{equation*}
\end{corollary} 

\quad For example, choosing $f(x)=x^{r+1}$ for some $r\geq-1$ allows us to prove
\begin{equation*}
\forall\,r>-1,\qquad\lim_{n\to\infty}\frac{1}{n^{r+2}}\sum_{k=1}^n k^{r}\alpha(k)=\frac{2}{3(r+2)},
\end{equation*}
and for $r=0$ we find that $U(n)\sim\frac{n^2}{3}$, but  in Section \ref{sec3} we will obtain far more interesting results about $U$.

Also, letting $f(x)=1/(x+a)$ for some $a>0$, yields
\begin{equation*}
\forall\,a>0, \lim_{n\to\infty}\sum_{k=1}^n \frac{\alpha(k)}{k(an+k)}=\frac{2}{3}\ln\left(1+\frac{1}{a}\right).
\end{equation*}
\bg
\quad It is interesting to study how $V(n)$ is distributed in the interval
$\big(\frac{2n}{3},\frac{2n+2}{3}\big)$, to this end we define the function $v$
for positive integers by
\begin{equation}\label{E:eq5}
\forall\,n\geq1,\qquad v(n)=V(n)-\frac{2n}{3},
\end{equation}
and we set $v(0)=0$ for convenience. 
In the next proposition we find some results concerning the function
$v$.

\bg
\begin{proposition}\label{p2} The function $v$ satisfies the following properties :

\begin{enumeratea} 
 \item  For each positive integer $n$, we have
\begin{equation*}
v(2n)=\frac12 v(n),\quad\hbox{ and } \quad v(2n+1)=\frac13+\frac12 v(n).
\end{equation*}
\label{itm21}
\item  If $n=(\eps_m\cdots\eps_{1}\eps_0)_2$ then
\begin{equation*}
 v(n)=\frac{1}{3}\sum_{k=0}^m\frac{\eps_k}{2^k}.
\end{equation*}
In particular, the set $\{v(n):n\geq1\}$ is a dense subset of the interval
$[0,\frac{2}{3}]$.\label{itm22}
\item Also, if  $n=(\eps_m\cdots\eps_{1}\eps_0)_2$ then
\begin{equation*}
 v(n)+\sum_{p=0}^mv(\floor{2^{-p}n})=\frac{2}{3}\sum_{k=0}^m\eps_k.
\end{equation*}\label{itm23}
\item (Symmetry) If $ n=(1\eps_{m-1}\cdots\eps_{1}\eps_0)_2\in I_m$, 
 and we set $\widehat{n}=(1\adh{\eps}_{m-1}\cdots\adh{\eps}_{1}\adh{\eps}_0)_2$,
 then $v(n)+v(\widehat{n})=\frac{2}{3}$\label{itm24}.
\end{enumeratea}
\end{proposition}

\begin{proof}
Indeed, \eqref{itm21} follows immediately from the recurrence relations for $V$ in Proposition \ref{p1}.

\smallskip
\quad Now, consider $n=(\eps_m\cdots\eps_{1}\eps_0)_2$. We have
\begin{equation*}
\floor{2^{-p}n}=\eps_p+\eps_{p+1}2+\cdots\eps_{m}2^{m-p}=\eps_p+2\floor{2^{-p-1}n},
\end{equation*}

So, using the recurrence relations in \eqref{itm21} we conclude that

\begin{equation}\label{E:eq6}
v\left(\floor{n2^{-p}}\right)=
v\left(\eps_p+2\floor{n2^{-p-1}}\right)=\frac{\eps_p}{3}+\frac{1}{2}v\left(\floor{n2^{-p-1}}\right),
\end{equation}
\smallskip
Multiplying both sides by $2^{-p}$ and adding the obtained relations as $p$ varies from $0$ to $m$ we find that\nobreak
\begin{equation*}
v(n)=\frac{1}{3}\sum_{p=0}^m\frac{\eps_p}{2^p},
\end{equation*}
which is the desired formula. This end the proof of \eqref{itm22} since the density statement is immediate. 

\quad On the other hand,
adding the equalities in \eqref{E:eq6} for $p\in\{0,1,\ldots,m\}$ we obtain
\begin{align*}
\sum_{p=0}^mv(\floor{2^{-p}n})&=\frac13\sum_{p=0}^m\eps_p
+\frac12\sum_{p=1}^mv(\floor{2^{-p}n})\\
&=\frac13\sum_{p=0}^m\eps_p-\frac12 v(n)
+\frac12\sum_{p=0}^mv(\floor{2^{-p}n})
\end{align*}
which is equivalent to \eqref{itm23}.

\quad Using \eqref{itm22} we can write
\begin{align*}
3v(\widehat{n})&=\frac{1}{2^m}+\sum_{k=0}^{m-1}\frac{1-\eps_k}{2^k}\\
&=\frac{1}{2^m}+2-\frac{2}{2^m}-\sum_{k=0}^{m-1}\frac{\eps_k}{2^k}\\
&=2-\left(\frac{1}{2^m}+\sum_{k=0}^{m-1}\frac{\eps_k}{2^k}\right)=2-3v(n),
\end{align*}
which is the desired symmetry result \eqref{itm24}. 
\end{proof}
\smallskip
\quad The following corollaries are immediate consequences :
\smallskip
\begin{corollary}\label{c3} 
For a positive integer $n$ we have
\begin{equation*}
\frac13<v(2n+1)<\frac23 \qquad\hbox{and}\qquad 0<v(2n)<\frac13.
\end{equation*}
\end{corollary}
\smallskip
\begin{corollary}\label{c4} 
Let $m$ be a nonnegative integer. Then for every $n\in I_m$ we have
\begin{equation*}
\frac{1}{3\cdot 2^m}\leq v(n)\leq \frac{2}{3}-\frac{2-2^{-m}}{3n},
\end{equation*} 
where the lower bound is attained if and only if $n=2^m$ and the upper bound is attained if and only if $n=2^{m+1}-1$.
\end{corollary}

\begin{proof}
Indeed,  for $n=(\eps_m\cdots\eps_{1}\eps_0)_2\in I_m$, using Proposition \ref{p2}\eqref{itm21}, we have,
\begin{equation*}
v(n)=\frac{1}{3}\sum_{p=0}^m\frac{\eps_p}{2^p}\geq\frac{\eps_m}{3\cdot 2^m}=\frac{1}{3\cdot 2^m},
\end{equation*}
with equality if and only if $\eps_0=\eps_1=\cdots=\eps_{m-1}=0$,
that is, if and only if $n=2^m$.\bg

The upper bound is a little bit trickier, consider
$n=(1\eps_{m-1}\cdots\eps_{1}\eps_0)_2\in I_m$ and recall that
$\widehat{n}=(1\adh{\eps}_{m-1}\cdots\adh{\eps}_{1}\adh{\eps}_0)_2$.
If $n< 2^{m+1}-1$ then there exists some $j$ in $\{0,\ldots,m-1\}$
such that $\eps_j=0$ and consequently, using Proposition\ref{p2}\eqref{itm21} again,
we find that
\begin{align*}
3nv(\widehat{n})&=\left(2^m+\sum_{k=0}^{m-1}\eps_k2^k\right)
\left(2^m+\sum_{k=0}^{m-1}\adh{\eps}_k2^{-k}\right)\\
&\geq 2^m \times 2^{-j}=2^{m-j}\geq 2.
\end{align*}
Whereas,  if $n=2^{m+1}-1$ then
\begin{equation*}
3nv(\widehat{n})=3(2^{m+1}-1)v(2^m)=2-2^{-m}<2.
\end{equation*}
Hence, we have shown that $3nv(\widehat{n})\geq 2-2^{-m}$ for every $n\in I_m$, with equality if and only if $n=2^{m+1}-1$. But, using Proposition \ref{p2}\eqref{itm24}, we have
$v(\widehat{n})=2/3-v(n)$, so, the above conclusion yields the desired
upper bound, and characterizes the case of equality.  
\end{proof}
\smallskip
{\bf Remark.} The upper bound obtained in Corollary \ref{c4} is sharper
than the one that could be obtained directly from Proposition \ref{p2}\eqref{itm22} which is $(2-2^{-m})/3$.

\bg
\quad Our final property for $V$ is the following result :

\smallskip
\begin{theorem}\label{t5}  For every positive integer $n$ we have
\begin{equation*}
\frac{2n^2+1}{3n}\leq V(n)\leq \frac{2n(n+2)}{3(n+1)}.
\end{equation*}
Moreover, the lower bound is attained if and only if $n=2^m$ for some
nonnegative integer $m$, and the upper bound is attained if
and only if $n=2^{m+1}-1$ for some nonnegative integer $m$.
\end{theorem}
\begin{proof} Consider $n\in I_m$. Since $n\geq2^m$ we conclude
using Corollary \ref{c4} that
$$V(n)-\frac{2n}{3}=v(n)\geq\frac{1}{3\cdot 2^m}\geq\frac{1}{3n}$$
with equality in both inequalities if and only if $n=2^m$. This proves the
first inequality and the characterizes the case of equality there.

\smallskip
\quad Let us come to the second inequality. Here we note that if
 $n\in I_m$ then $n+1\leq2^{m+1}$, So, again, using Corollary \ref{c4} we have
$$V(n)-\frac{2n}{3}=v(n)\leq
\frac{2}{3}\left(1-\frac{1}{n}+\frac{1}{n2^{m+1}}\right)
\leq\frac{2}{3}\left(1-\frac{1}{n}+\frac{1}{n(n+1)}\right)
$$
with equality in both inequalities if and only if $n=2^{m+1}-1$. This proves the second inequality and the characterizes the case of equality there.
\end{proof}
\bg

\section{\bf The Properties of U }\label{sec3}

\smallskip
\quad Let us start by considering the recurrence relation satisfied by the sum $U$ defined by formula \eqref{E:eq3}.

\smallskip
\begin{proposition}\label{p3}
For every nonnegative integer $n$ we have
\begin{equation*}
U(2n)=n^2+U(n) \quad\hbox{and}\quad U(2n+1)=(n+1)^2+U(n)
\end{equation*}
with the convention $U(0)=0$.
\end{proposition}
\begin{proof}
Indeed, for a positive integer $n$ we have
\begin{align*}
U(2n)&=\sum_{k=1}^{n}\alpha(2k-1)+\sum_{k=1}^{n}\alpha(2k)\\
&=\sum_{k=1}^{n}(2k-1)+\sum_{k=1}^{n}\alpha(k)=n^2+U(n).
\end{align*}
Also
\begin{align*}
U(2n+1)&=\alpha(2n+1)+U(2n)\\
&=2n+1+n^2+U(n)=(n+1)^2+U(n).
\end{align*}
Clearly, the conclusion holds also for $n=0$.
\end{proof}

\smallskip
\quad Before stating the main result concerning $U$, let us prove
the following lemma :

\smallskip
\begin{lemma}\label{l1}
For a positive integer $n\in I_m$ with binary representation $n=(\eps_m\cdots\eps_0)_2$, we define $h(n)$ by the formula
\begin{equation*}
h(n)=\sum_{k=0}^{m-1}\adh{\eps}_k\floor{\frac{n}{2^{k+1}}}.
\end{equation*}
Then we have $0\leq h(n)\leq n-1$. 
Moreover, $h(n)=0$ if and only if $n=2^{m+1}-1$, and
$h(n)=n-1$ if and only if $n=2^m$.
\end{lemma} 

\bg
\begin{proof} Clearly, we have $h(2^{m+1}-1)=0$. Now,
if  $(\eps_m\cdots\eps_0)_2$ is the binary representation of some $n\in I_m$ satisfying $n<2^{m+1}-1$ then there must be some
$j\in\{0,\ldots,m-1\}$ such that $\eps_j=0$. In this case we have
\begin{equation*}
h(n)\geq \adh{\eps}_j\floor{\frac{n}{2^{j+1}}}\geq
\floor{\frac{2^m}{2^{j+1}}}=2^{m-j-1}\geq1
\end{equation*}

So, we have proved the first inequality and characterized the case of
equality in it.

\bg
On the other hand, we have
\begin{equation*}
h(2^m)=\sum_{k=0}^{m-1}2^{m-k-1}=2^m-1,
\end{equation*}
and for $n\in I_m$ we can write
\begin{equation*}
h(n)\leq\sum_{k=0}^{m-1}\frac{n}{2^{k+1}}=n(1-2^{-m})=n-1-\frac{n-2^m}{2^m}.
\end{equation*}
Therefore, if $n\in I_m$ satisfies $n>2^m$ then $h(n)<n-1$. This
achieves the proof of the lemma.
\end{proof}

\smallskip
\quad Now, we come to our main result concerning the sum $U$.

\begin{theorem}\label{t2}
For every positive integer $n$ the following is true :
\begin{enumeratea}
\item If $n$ is even then
\begin{equation*}
\frac{n^2+2}{3}\leq U(n)\leq\frac{n^2+n}{3},
\end{equation*} 
with equality in the lower bound if and only if $n=2^m$ for some positive integer $m$, and equality in the upper bound if and only if $n=2^{m}-2$ for some positive integer $m$.\label{itmt21}
\item If $n$ is odd then
\begin{equation*}
\frac{n^2+n+1}{3}\leq U(n)\leq\frac{n^2+2n}{3},
\end{equation*} 
with equality in the lower bound if and only if $n=2^m+1$ for some positive integer $m$, and equality in the upper bound if and only if $n=2^{m}-1$  for some positive integer $m$.\label{itmt22}
\end{enumeratea}
\end{theorem}

\begin{proof}
For a nonnegative integer $n$ we define $u(n)$ by
\begin{equation}
u(n)=\frac{n^2+n}{3}-U(n).
\end{equation}
Clearly, using Proposition \ref{p3}, we have
\begin{align*}
u(2n)&=\frac{4n^2+2n}{3}-n^2-U(n)=\frac{n}{3}+u(n),\\
\noalign{\hbox{and}}\\
u(2n+1)&=\frac{(2n+1)^2+2n+1}{3}-(n+1)^2-U(n)
=-\frac{n+1}{3}+u(n).
\end{align*}
We can express the above two formulas as follows
\begin{equation}\label{E:eq8}
u(2n+\eps)=u(n)+\frac{n}{3}-\frac{2n+\eps}{3}\eps,
\end{equation}
for every nonnegative integer $n$ and every $\eps\in\{0,1\}$.

\smallskip
\quad Now, consider $n=(\eps_m\cdots\eps_{1}\eps_0)_2\in I_m$. Since
$\floor{2^{-p}n}=\eps_p+2\floor{2^{-p-1}n}$, we conclude, using \eqref{E:eq8}, that
\begin{equation*}
u\left(\floor{n2^{-p}}\right)=u\left(\floor{n2^{-p-1}}\right)+
\frac{1}{3}\floor{n2^{-p-1}}-\frac{1}{3}\eps_p\floor{n2^{-p}}.
\end{equation*}

Adding these equations as $p$ varies in $\{0,1,\ldots,m\}$ we find that
\begin{align*}
u(n)&=\frac{1}{3}\sum_{p=1}^m\floor{n2^{-p}}
-\frac{1}{3}\sum_{p=0}^m\eps_p\floor{n2^{-p}}\\
&=-\frac{\eps_0\,n}{3}+\frac{1}{3}\sum_{p=1}^{m-1}\adh{\eps}_p\floor{n2^{-p}}\\
&=-\frac{\eps_0\,n}{3}+\frac{2}{3}\sum_{p=1}^{m-1}\adh{\eps}_p\floor{n2^{-p-1}},
\end{align*}
where the last equality follows from the fact that $\adh{\eps}_p\eps_p=0$ for every $p$. Thus, we have shown that for
$n=(\eps_m\cdots\eps_{1}\eps_0)_2\in I_m$ the following holds
\begin{equation*}
u(n)=-\frac{\eps_0\,n}{3}+\frac{2}{3}h\left(\floor{\frac{n}{2}}\right),
\end{equation*}
where $h$ is the function defined in Lemma \ref{l1}.

 Let us discuss the following two cases :

\smallskip
\begin{itemize}
\item $n$ is even. In this case $\eps_0=0$ and $u(n)=\frac{2}{3}h(n/2)$. By Lemma \ref{l1} we conclude that 
$$ 0\leq u(n) \leq \frac{n-2}{3},$$
with equality in the first inequality if and only if $n=2(2^m-1)$, and equality in the second inequality if and only if $n=2(2^{m-1})$. This is equivalent to the desired conclusion and achieves the proof of part \eqref{itmt21}.

\smallskip
\item $n$ is odd. In this case $\eps_0=1$ and $u(n)=-\frac{1}{3}n+\frac{2}{3}h((n-1)/2)$. By Lemma \ref{l1} we conclude that 
$$0\leq u(n)+ \frac{n}{3} \leq \frac{n}{3}-1,$$
with equality in the first inequality if and only if $n=2(2^m-1)+1$, and equality in the second inequality if and only if $n=2(2^{m-1})+1$.  This is equivalent to the desired conclusion and achieves the proof of part \eqref{itmt22}.
\end{itemize}
The proof of the theorem is complete.
\end{proof}

\smallskip
\section{\bf The Properties of G }\label{sec4}

\smallskip
\quad Now, we come to the function $G$ defined in by the formula \eqref{E:eq4}. We seek sharp bounds for the values of $G(n)$. The
following Proposition gives such bounds, and it constitutes a refinement
upon the bounds in \cite{ben} :

\bg
\begin{proposition}\label{p4} 
\begin{enumeratea}
\item  For each positive integer $n$, we have
\begin{equation*}
G(2n)=n(n+1)+G(n)-\frac12 V(n)\quad\hbox{and}\quad G(2n+1)=(n+1)^2+G(n).
\end{equation*}
where $V$ is the function defined in \eqref{E:eq1}.\label{itmp41}
\item For each positive integer $n$, we have
\begin{equation*}
\frac{n(n+4/7)}{3}\leq G(n)\leq\frac{n(n+2)}{3}.
\end{equation*}\label{itmp42}
\end{enumeratea}
\end{proposition}

\smallskip
\begin{proof}
Clearly we have,
\begin{align*}
G(2n+1)&=\sum_{k=1}^{2n+1}(2n+2-k)\frac{\alpha(k)}{k}\\
&=\sum_{k=0}^{n}(2n+2-2k-1)\frac{\alpha(2k+1)}{2k+1}+\sum_{k=1}^{n}(2n+2-2k)\frac{\alpha(2k)}{2k}\\
&=\sum_{k=0}^{n}(2(n-k)+1)+\sum_{k=1}^{n}(n+1-k)\frac{\alpha(k)}{k}\\
&=(n+1)^2+G(n),\\
\noalign{\hbox{and}}\\
G(2n)&=\sum_{k=1}^{2n}(2n+1-k)\frac{\alpha(k)}{k}\\
&=\sum_{k=1}^{n}(2n+1-2k+1)\frac{\alpha(2k-1)}{2k-1}+\sum_{k=1}^{n}(2n+1-2k)\frac{\alpha(2k)}{2k}\\
&=\sum_{k=1}^{n}2(n-k+1)+\sum_{k=1}^{n}(n+1-k)\frac{\alpha(k)}{k}-\frac12\sum_{k=1}^{n}\frac{\alpha(k)}{k}\\
&=n(n+1)+G(n)-\frac12V(n),
\end{align*}
This proves \eqref{itmp41}.

\smallskip
\quad Now, we will prove by induction on $m$ the following property :
$$\mathcal{R}_m:\qquad \forall\,n\in I_m,\quad \frac{n(n+7/4)}{3}\leq G(n)\leq\frac{n(n+2)}{3}.$$
First, it is straightforward to check that

\smallskip
\centerline{ \vbox{\offinterlineskip
\halign{\strut#&\vrule#&\quad\hfil$#$\hfil\quad&\vrule#&\hfil\quad$#$\quad\hfil&\quad\hfil$#$\hfil\quad&\hfil\quad$#$\quad\hfil&\hfil\quad$#$\quad\hfil&
\hfil\quad$#$\quad\hfil&\hfil\quad$#$\quad\hfil&\hfil\quad$#$\quad\hfil&\vrule#\cr
\noalign{\hrule}
\omit&height2pt&&&&&&&&&&\cr
&& n &&1&2&3&4&5& 6&7& \cr
\omit&height2pt&&&&&&&&&&\cr
\noalign{\hrule}
\omit&height2pt&&&&&&&&&&\cr
&&\frac{n(n+2)}{3}-G(n) &&0&\frac16&0&\frac14&\frac16&\frac14&0& \cr
\omit&height2pt&&&&&&&&&&\cr
\noalign{\hrule}
\omit&height2pt&&&&&&&&&&\cr
&&G(n)-\frac{n(n+7/4)}{3} &&\frac{1}{12}&0&\frac14&\frac{1}{12}&\frac14&\frac14&\frac{7}{12} &\cr
\omit&height2pt&&&&&&&&&&\cr
\noalign{\hrule}
}}
}

\smallskip
So $\mathcal{R}_m$ is true for $m=1,2,3$. Let us suppose that
$\mathcal{R}_m$ is true for some $m\geq3$, and consider $n\in I_{m+1}$. 

\smallskip\goodbreak

There are two cases :
\begin{itemize}
\item  $n=2p$ for some $p\in I_m$. Then
$\frac{p(p+7/4)}{3}\leq G(p)\leq\frac{p(p+2)}{3}$ by the
induction hypothesis, and $\frac{2p}{3}<V(p)<\frac{2p+2}{3}$ by Proposition \ref{p1}\eqref{itm12}. Hence
\begin{align*}
G(n)&=G(2p)=p(p+1)+G(p)-\frac12V(p)\\
&\leq p(p+1)+\frac{p(p+2)}{3}-\frac{p}{3}=\frac{4p(p+1)}{3}=\frac{n(n+2)}{3},\\
\noalign{\qquad\,\hbox{and}}\\
G(n)&\geq p(p+1)+\frac{p(p+7/4)}{3}-\frac{p+1}{3}\\
&=\frac{4p^2+\frac{15}{4}p-1}{3}=\frac{n(n+7/4)}{3}+\frac{n-8}{24}\geq
\frac{n(n+7/4)}{3},
\end{align*}
where we used the fact that for $m\geq3$ we have $n\geq8$.

\bg
\item  $n=2p+1$ for some $p\in I_m$.
Then $\frac{p(p+7/4)}{3}\leq G(p)\leq\frac{p(p+2)}{3}$, hence
\begin{align*}
G(n)&=G(2p+1)=(p+1)^2+G(p)\\
&\leq (p+1)^2+\frac{p(p+2)}{3}\\
&=\frac{4p^2+8p+3}{3}=\frac{n(n+2)}{3},\\
\noalign{\qquad\,\hbox{and}}\\
G(n)&\geq (p+1)^2+\frac{p(p+7/4)}{3}\\
&=\frac{4p^2+\frac{31}{4}p+3}{3}=\frac{n(n+7/4)}{3}+\frac{p+1}{4}\geq
\frac{n(n+7/4)}{3},
\end{align*}
\end{itemize}

We conclude that $\frac{n(n+7/4)}{3}\leq G(n)\leq\frac{n(n+2)}{3}$ for every $n\in I_{m+1}$. This achieves the proof of the induction step : $\mathcal{R}_m\Longrightarrow \mathcal{R}_{m+1}$ for $m\geq 3$,
and completes the proof of \eqref{itmp42}.
\end{proof}

\bg
\quad It seems that the values of $G(n)$ become closer and closer to the upper bound given in Proposition \ref{p4}.
In order to study this property, we consider the function $g$ defined for nonnegative integers by
\begin{equation}\label{E:eq16}
g(n)=\frac{n(n+2)}{3}-G(n)
\end{equation} 
with the convention $g(0)=0$. The following proposition gives some properties of $g$.

\bg
\begin{proposition}\label{p5}
\begin{enumeratea}
\item  For each positive integer $n$, we have
\begin{equation*}
g(2n)=g(n)+\frac12 v(n)\quad\hbox{and}\quad g(2n+1)=g(n).
\end{equation*}
where $v$ is the function by the formula \eqref{E:eq5}.\label{itmp51}
\item  For each positive integer $n$, if $ n=(\eps_m,\ldots,\eps_1,\eps_0)_2$, then 
\begin{equation*}
 g(n)=\frac{1}{2}\sum_{p\geq 0}\adh{\eps}_pv(\floor{2^{-p-1}n}).
\end{equation*}\label{itmp52}
\item  For each positive integer $n$, we have
$0\leq g(n)\leq\frac{1}{3}\floor{\lg n}$.\label{itmp53}
\end{enumeratea}
\end{proposition}

\begin{proof} Indeed, using the recurrence relations for $G$, (see Proposition \ref{p4}\eqref{itmp41},) we can write
\begin{align*}
g(2n+1)&=\frac{(2n+1)(2n+3)}{3}-G(2n+1)\\
&=\frac{(2n+1)(2n+3)}{3}-(n+1)^2-G(n)\\
&=\frac{n(n+2)}{3}-G(n)=g(n)\\
\noalign{\hbox{and}}\\
g(2n)&=\frac{2n(2n+2)}{3}-G(2n)\\
&=\frac{4n(n+1)}{3}-n(n+1)-G(n)+\frac12 V(n)\\
&=\frac{n(n+2)}{3}-G(n)+\frac12 V(n)-\frac{n}{3}=g(n)+\frac12 v(n)
\end{align*}
This proves \eqref{itmp51}.

\smallskip
\quad Recalling that $\floor{2^{-p}n}=\eps_p+2\floor{2^{-p-1}n}$, we deduce from the recurrence
relations in \eqref{itmp51} that
\begin{equation*}
g(\floor{2^{-p}n})-g(\floor{2^{-p-1}n})=\frac{\adh{\eps}_p}{2} v(\floor{2^{-p-1}n}).
\end{equation*}
Adding these equalities for $p\in\{0,1,\ldots,m\}$ we find that
$$g(n)=\frac12\sum_{p=0}^{m-1}\adh{\eps}_pv(\floor{2^{-p-1}n}),$$
which is \eqref{itmp52}.

\smallskip
\quad Finally, using Proposition \ref{p2} \eqref{itm22} we see that $v$ takes its values in $[0,2/3]$, and consequently
$$\forall\,n\geq 1,\quad 0\leq g(n)\leq \frac{m}{3}=\frac{1}{3}\floor{\lg n}.$$
which is \eqref{itmp53}. This completes the proof.
\end{proof}
\bg

\quad We have seen that $n(n+2)/3$ is an upper bound for $G(n)$. In the next corollary we will show that
this upper bound is attained infinitely many times, more precisely we will prove the following :

\smallskip
\begin{corollary}\label{c5}
We have 
\begin{equation*}
\left\{n\geq1:G(n)=\frac{n(n+2)}{3}\right\}=\big\{2^r-1:r\geq1\big\}.
\end{equation*}
\end{corollary}

\smallskip
\begin{proof}  
\quad By \eqref{E:eq16}, we are looking for the set of positive integers $n$ such that $g(n)=0$.
Consider $ n=(\eps_m\cdots\eps_1\eps_0)_2\in I_m$, that is $\eps_m=1$. The case $m=0$ corresponds to $n=1=2^1-1$ and we know that $g(1)=0$, so let us suppose
that $m\geq1$. 

By Proposition \ref{p5}\eqref{itmp52} we have 
\begin{equation*}
 g(n)=\frac{1}{2}\sum_{p=0}^{m-1}\adh{\eps}_pv(\floor{2^{-p-1}n})
\end{equation*}
But, for $p\in\{0,1,\ldots,m-1\}$, we have $2^{-p-1}n\geq 2^{m-p-1}\geq1$ so $\floor{2^{-p-1}n}>0$,
and consequently $v(\floor{2^{-p-1}n})>0$, by Corollary \ref{c3}. It follows that $g(n)=0$ if and only if
$\adh{\eps}_p=0$  for $p\in\{0,1,\ldots,m-1\}$, or equivalently $\eps_p=1$ for $p\in\{0,1,\ldots,m\}$. That
is $n=\sum_{p=0}^m2^p=2^{m+1}-1$.
\end{proof}

\smallskip
\quad Now, we will introduce a symmetry property satisfied by $g$.
\begin{proposition}\label{p6}
\begin{enumeratea}
\item For each positive integer $n$, if $n=(1\eps_{m-1}\cdots\eps_10)_2\in I_m$,  and if 
$\tilde{n}=(1\adh{\eps}_{m-1}\cdots\adh{\eps}_10)_2$, then
$ g(n)=g(\tilde{n})$.\label{itmp61}
\item More generally, for each positive integer $n$, if  $\tilde{n}=3\cdot 2^{\floor{\lg n}}-2-n$, then we have
$g(n)=g(\tilde{n})$.\label{itmp62}
\end{enumeratea}
\end{proposition}

\begin{proof}
 Indeed, by Proposition \ref{p5}\eqref{itmp52} we have
\begin{equation*}
g(n)=\frac12 v\left(\frac{n}{2}\right)+\frac12\sum_{k=1}^{m-1}\adh{\eps}_k v(\floor{2^{-k-1}n}).
\end{equation*}
Now, using Proposition \ref{p2}\eqref{itm24} we see that, for every $k\in\{0,\ldots,m-1\}$, we have $v(\floor{2^{-k-1}n})=\frac23-v(\floor{2^{-k-1}\tilde{n}})$.
So,
\begin{align*}
g(n)&=\frac12 v\left(\frac{n}{2}\right)+\frac12\sum_{k=1}^{m-1} v(\floor{2^{-k-1}n})-\frac12\sum_{k=1}^{m-1}\eps_k v(\floor{2^{-k-1}n})\\
&=\frac12 v\left(\frac{n}{2}\right)+\frac12\sum_{k=1}^{m-1} v(\floor{2^{-k-1}n})-\frac13\sum_{k=1}^{m-1}\eps_k +\frac12\sum_{k=1}^{m-1}\eps_k v(\floor{2^{-k-1}\tilde{n}}),
\end{align*}

and
\begin{align*}
g(\tilde{n})&=\frac12 v\left(\frac{\tilde{n}}{2}\right)+\frac12\sum_{k=1}^{m-1}{\eps}_k v(\floor{2^{-k-1}\tilde{n}})\\
&=\frac13-\frac12 v\left(\frac{n}{2}\right)+\frac12\sum_{k=1}^{m-1}{\eps}_k v(\floor{2^{-k-1}\tilde{n}}).
\end{align*}

Hence, with $\eps_m=1$, we have
\begin{align*}
g(n)-g(\tilde{n})&= v\left(\frac{n}{2}\right)
+\frac12\sum_{k=1}^{m-1} v(\floor{2^{-k-1}n})-\frac13\sum_{k=1}^{m-1}\eps_k-\frac13\\
&=\frac12\left(v\left(\frac{n}{2}\right)+
\sum_{k=0}^{m-1} v\left(\floor{2^{-k}\frac{n}{2}}\right)
-\frac23\sum_{k=1}^{m}\eps_k\right)=0.
\end{align*}
where we used Proposition \ref{p2}\eqref{itm23}. This ends the proof of \eqref{itmp61}.

\smallskip
\quad Now, consider a positive integer $n$ and let $m=\floor{\lg n}$. We have $ n\in I_m$, and $n$ has the binary representation 
$n=(1\eps_{m-1}\cdots\eps_1\eps_0)_2$, with
$\eps_k\in\{0,1\}$. There are two cases:

\smallskip
\begin{itemize}
\item  $\eps_0=0$. In this case we have
\begin{equation*}
3\cdot2^m-2-n=\sum_{k=1}^m2^k-\sum_{k=1}^{m-1}\eps_k2^k=2^m+\sum_{k=1}^{m-1}\adh{\eps}_k2^k=\tilde{n},
\end{equation*}
and \eqref{itmp61} is equivalent to $g(n)=g(3\cdot 2^{\floor{\lg n}}-2-n)$ in this case.

\smallskip
\item $\eps_0=1$. Here, we consider also two cases :

\begin{itemize}
\item For every $k\in\{0,1,\ldots,m-1\},~\eps_k=1$. In this case
we have $n=2^{m+1}-1$ and $3\cdot2^m-2-n=2^m-1$, and we have seen that
$g(2^r-1)=0$ for every $r$, so $g(n)=g(3\cdot 2^{\floor{\lg n}}-2-n)$ in this case also.

\smallskip
\item There exists $k\in\{1,\ldots,m-1\}$, such that $\eps_k=0$. In this case we
define $j=\min\{k\in\{1,\ldots,m-1\},~\eps_k=0\}$ so that
\begin{align*}
n&=2^m+\sum_{j<k<m}\eps_{k}2^k+\sum_{0\leq k<j}2^k\\
&=2^m+2^j-1+\sum_{j< k<m}\eps_{k}2^k= p2^j+2^j-1.
\end{align*}
with $p=2^{m-j}+\sum_{j< k<m}\eps_{k}2^{k-j}$, and
\begin{align*}
3\cdot2^m-2-n&=2^{m+1}-2^j-1-\sum_{j\leq k<m}\eps_{k}2^k\\
&=2^{m}+2^j-1+\sum_{j< k<m}(1-\eps_{k})2^k\\
&=2^{m}+2^j-1+\sum_{j< k<m}\adh{\eps}_{k}2^k = \tilde{p}2^j+2^j-1.
\end{align*}
\quad Now, using the fact that $g(2p+1)=g(p)$ repeatedly we see that
$g(p)=g(p2^j+2^j-1)$ for every $j$ and $p$. Therefore, using part \eqref{itmp61}, we obtain
\begin{equation*}
g(n) =g(p2^j+2^j-1)=g(p)=g(\tilde{p})=g(\tilde{p}2^j+2^j-1)=g(3\cdot2^m-2-n).
\end{equation*}
\end{itemize}
\end{itemize}
This completes the proof of Proposition \ref{p6}.
\end{proof}
\bg
\quad In Proposition \ref{p4}\eqref{itmp42} we have proved that
\begin{equation*}
\frac{n(n+7/4)}{3}\leq G(n),
\end{equation*}
but this inequality is not sharp for large values of $n$, since by Proposition \ref{p5}\eqref{itmp53} we have
\begin{equation*}
\frac{n(n+2)}{3}-\frac{\floor{\lg n}}{3}\leq G(n),
\end{equation*}
or equivalently
\begin{equation*}
\frac{n(n+2-\floor{\lg n}/n)}{3}\leq G(n).
\end{equation*}
Unfortunately, this inequality is again not sharp enough. Our next objective is to find
a sharp inequality, where equality holds infinitely many times. To this end we will need some 
preliminary results.

\smallskip

\quad For a nonnegative integer $r$ we consider $x_r$ and $y_r$ defined by
\begin{equation}\label{E:eq41}
x_r=\sum_{0\leq k<r}^{r-1} 2^{2k+1}=\frac{2}{3}(2^{2r}-1),\quad\hbox{and}\quad y_r=2x_r.
\end{equation}
Clearly we have $\floor{\lg x_r}=2r-1$ and $\floor{\lg y_r}=2r$ for $r>0$, and
\begin{equation*}
x_r=(\underbrace{1010\cdots10}_{2r \hbox{\SMALL{ digits}}})_2\qquad\hbox{and}\qquad
y_r=(\underbrace{1010\cdots100}_{2r+1 \hbox{\SMALL{ digits}}})_2.
\end{equation*}
Also, $x_r$ and $y_r$ can be defined by the recurrence relations :
\begin{equation}\label{E:eq42}
x_0=y_0=0,\qquad x_{r+1}=4x_r+2,\quad y_{r+1}=4y_r+4.
\end{equation}

These sequences of integers will play an important role in the sequal.

Clearly, we have
\begin{equation*}
v(x_r)=v\left(\sum_{k=0}^{r-1} 2^{2k+1}\right)=\frac13\sum_{k=0}^{r-1}\frac{1}{2^{2k+1}}=\frac29-\frac{2}{9\cdot2^{2r}},
\end{equation*}
and since $v(y_r)=v(2x_r)=\frac12v(x_r)$ we conclude that
\begin{equation}\label{E:eq43}
v(x_r)=\frac29-\frac{2}{9\cdot2^{2r}} \qquad\hbox{and}\qquad v(y_r)=\frac19-\frac{1}{9\cdot2^{2r}}.
\end{equation}

\quad Now, let us prove a technical result about $g$.

\smallskip
\begin{lemma}\label{l2}
For every nonnegative integers $p$ and $r$ we have
\begin{align*}
g(2^{2r+2}p+x_{r+1})-g(2^{2r+2}p+y_{r})&=\frac13 \left(1+\frac{1}{2^{2r+1}}\right)\left(\frac13-v(p)\right)
\intertext{and}
g(2^{2r+1}p+x_{r})-g(2^{2r+1}p+y_{r})&=\frac13 \left(1-\frac{1}{2^{2r}}\right)\left(v(p)-\frac13\right)
\end{align*}
\end{lemma}
\begin{proof}
Using the recurrence relations for $g$ from Proposition \ref{p5}\eqref{itmp51} we deduce immediately
the following ``two-stage''  recurrence relations, which are valid for every nonnegative integer $n$ :

\begin{equation}\label{E:eql21}
\begin{aligned}
g(4n)&=g(2n)+\frac12v(2n)=g(n)+\frac12v(n)+\frac14v(n)=g(n)+\frac34v(n),\\
g(4n+1)&=g(2n)=g(n)+\frac12v(n),\\
g(4n+2)&=g(2n+1)+\frac12v(2n+1)=g(n)+\frac16+\frac14v(n),\\
g(4n+3)&=g(2n+1)=g(n).
\end{aligned}
\end{equation}

It follows that for nonnegative integers $n$ and $x$ we have
\begin{align*}
g(8n+8x+4)&=g(2n+2x+1)+\frac34v(2n+2x+1)\\
&=g(n+x)+\frac34v(2n+2x+1)\\
&=g(2n+2x)-v(2n+2x)+\frac14+\frac34v(2n+2x)\\
&=g(2n+2x)+\frac14-\frac14 v(2n+2x),
\end{align*}
applying this with $n=2^{2k}p$ and $x=x_k$ we find that
\begin{equation}\label{E:eql22}
g(2^{2k+3}p+y_{k+1})=g(2^{2k+1}p+y_{k})+\frac14-\frac14v(2^{2k+1}p+y_k).
\end{equation}
But, since $y_k<2^{2k+1}$ we conclude, using Proposition \ref{p2} and \eqref{E:eq43}, that
\begin{equation*}
v(2^{2k+1}p+y_k)=v(2^{2k+1}p)+v(y_k)=\frac{1}{2^{2k+1}}v(p)+\frac19-\frac{1}{9\cdot2^{2k}},
\end{equation*}
so we can rewrite \eqref{E:eql22} as follows
\begin{equation*}
g(2^{2k+3}p+y_{k+1})-g(2^{2k+1}p+y_{k})
=\frac29+\left(\frac29-v(p)\right)\frac{1}{2^{2k+3}}.
\end{equation*}
Adding these equalities as $k$ varies from $0$ to $r-1$  for some $r\geq 1$, we find that
\begin{equation*}
g(2^{2r+1}p+y_{r})-g(2p)=\frac{2r}{9}+\left(\frac{1}{27}-\frac{v(p)}{6}\right)\left(1-\frac{1}{2^{2r}}\right),
\end{equation*}
which is also true for $r=0$. This equivalent to
\begin{equation}\label{E:eql23}
g(2^{2r+1}p+y_{r})=
g(p)+\frac13 \left(1+\frac{1}{2^{2r+1}}\right)v(p)+\frac{2r}{9}+\frac{1}{27}\left(1-\frac{1}{2^{2r}}\right).
\end{equation}
In particular, taking $p=0$ we find
\begin{equation}\label{E:eql24}
g(y_{r})=\frac{2r}{9}+\frac{1}{27}\left(1-\frac{1}{2^{2r}}\right),
\end{equation}
and we can reformulate \eqref{E:eql23} as follows :
\begin{equation}\label{E:eql25}
g(2^{2r+1}p+y_{r})= g(p)+\frac13 \left(1+\frac{1}{2^{2r+1}}\right)v(p)+g(y_r).
\end{equation}
Also, recalling that $g(n)=g(2n)-v(2n)$ by Proposition \ref{p5} we conclude from \eqref{E:eql25} that 
\begin{align}\label{E:eql26}
g(2^{2r}p+x_{r})&=g(2^{2r+1}p+y_{r})-v(2^{2r+1}p+y_{r})\notag\\
&= g(p)+\frac13 \left(1+\frac{1}{2^{2r+1}}\right)v(p)+g(y_r)-\frac{1}{2^{2r+1}}v(p)-v(y_r)\notag\\
&= g(p)+\frac13 \left(1-\frac{1}{2^{2r}}\right)v(p)+g(y_r)-v(y_r)\notag\\
&= g(p)+\frac13 \left(1-\frac{1}{2^{2r}}\right)v(p)+g(x_r).
\end{align}
Replacing $p$ by $2p$, and using the recurrence relations from Proposition \ref{p5} and Proposition \ref{p2}, we find that
\begin{equation}\label{E:eql27}
g(2^{2r+1}p+x_{r})=g(p)+\frac13 \left(2-\frac{1}{2^{2r+1}}\right)v(p)+g(x_r).
\end{equation}
Also, replacing $p$ by $2p$ in \eqref{E:eql25} yields
\begin{align}\label{E:eql28}
g(2^{2r+2}p+y_{r})&= g(2p)+\frac13 \left(1+\frac{1}{2^{2r+1}}\right)v(2p)+g(y_{r})\notag \\
&= g(p)+\frac13 \left(2+\frac{1}{2^{2r+2}}\right)v(p)+g(y_{r}).
\end{align}
Now, using \eqref{E:eql24} and \eqref{E:eq43} we get
\begin{align*}
g(x_{r+1})-g(y_{r})&=g(2x_{r+1})-v(2x_{r+1})-g(y_{r})\\
&=g(y_{r+1})-g(y_{r})-v(y_{r+1})\\
&=\frac{1}{9}+\frac{1}{9\cdot2^{2r+1}}.
\end{align*}
Hence, from \eqref{E:eql28} and \eqref{E:eql26} with $r$ replaced by $r+1$ we obtain
\begin{align*}
g(2^{2r+2}p+x_{r+1})-g(2^{2r+2}p+y_{r})&=\frac13 \left(-1-\frac{1}{2^{2r+1}}\right)v(p)+g(x_{r+1})-g(y_{r})\\
&=\frac13 \left(1+\frac{1}{2^{2r+1}}\right)\left(\frac13-v(p)\right).
\end{align*}
Which is the first identity in the Lemma. 
\bg
Similarly, since $g(x_r)-g(y_r)=-v(y_r)$, we conclude from \eqref{E:eql25} and \eqref{E:eql27} that
\begin{align*}
g(2^{2r+1}p+x_{r})-g(2^{2r+1}p+y_{r})&=
\frac13 \left(1-\frac{1}{2^{2r}}\right)v(p)+g(x_r)-g(y_r)\\
&=\frac13 \left(1-\frac{1}{2^{2r}}\right)v(p)-\frac19+\frac{1}{9\cdot2^{2r}}\\
&=\frac13 \left(1-\frac{1}{2^{2r}}\right)\left(v(p)-\frac13\right).
\end{align*}
Which is the second identity in the Lemma. This achieves the proof of the lemma.
\end{proof}
\bg
\quad The next corollary is an immediate consequence of Lemma \ref{l2} and Corollary \ref{c3}.

\begin{corollary}\label{c6}
For positive integers $p$ and $r$,  the following inequalities hold
\begin{align*}
g(2^{2r+2}p+x_{r})&< g(2^{2r+2}p+y_{r}),\\
g(2^{2r+2}p+2^{2r+1}+y_{r})&< g(2^{2r+2}p+x_{r+1}),\\
g(2^{2r+1}p+y_{r-1}) &< g(2^{2r+1}p+x_{r}),\\
g(2^{2r+1}p+2^{2r}+x_{r})&< g(2^{2r+1}p+y_{r}).
\end{align*}
\end{corollary}

\quad Corollary \ref{c6} is the main tool for proving the following interesting theorem.

\begin{theorem}\label{t3} 
For positive integers $m$ and $n$, we define $\Lambda(n,m)$  by
\begin{equation*}
\Lambda(n,m)=\max\left(g(2^mn+t):0\leq t\leq 2^m-1\right).
\end{equation*}
Then for any positive integer $n$ and any nonnegative integer $m$ we have:
\begin{align*}
\Lambda(n,2m+1)&=\max\big(g(2^{2m+1}n+y_m),g(2^{2m+1}n+x_m)\big),\cr
\Lambda(n,2m+2)&=\max\big(g(2^{2m+2}n+y_{m}),g(2^{2m+2}n+x_{m+1})\big).
\end{align*}
\end{theorem}
\bg
\begin{proof}
Clearly, Since $g(2n+1)=g(n)$ and $g(2n)=g(n)+v(2n)$ we have
\begin{equation*}
\Lambda(n,1)=\max(g(2n),g(2n+1))=g(2n).
\end{equation*}
Also, in view of the recurrence relations in \eqref{E:eql21} we have
\begin{equation*}
\Lambda(n,2)=\max(g(4n),g(4n+1),g(4n+2),g(4n+3))=\max(g(4n),g(4n+2)).
\end{equation*}
Therefore, the conclusion of the theorem is trivially true for $m=0$,  since $x_0=y_0=0$ and $x_1=2$.

\smallskip
Generally, since we have
\begin{align*}
\Lambda(2n,m)&=\max\left(g(2^{m+1}n+t):0\leq t\leq 2^m-1\right),\\
\intertext{and}
\Lambda(2n+1,m)&=\max\left(g(2^{m+1}n+t): 2^m\leq t\leq 2^{m+1}-1\right),
\end{align*}
we see immediately that
\begin{equation}\label{E:eqt31}
\Lambda(n,m+1)=\max\big(\Lambda(2n,m),\Lambda(2n+1,m)\big).
\end{equation}

\smallskip
\quad Let us proceed by induction on $m$. The base case of $m=0$ is trivially true according
to what we have shown earlier.

\quad Suppose that the result is true for $m-1$ for some $m\geq 1$, then, using Corollary \ref{c6} we have
\begin{align*}
\Lambda(2n,2m)&=\max\big(g(2^{2m+1}n+y_{m-1}),g(2^{2m+1}n+x_{m})\big)\\
&=g(2^{2m+1}n+x_{m})\\
\Lambda(2n+1,2m)&=\max\big(g(2^{2m+1}n+y_{m}),g(2^{2m+1}n+2^{2m}+x_{m})\big)\\
&=g(2^{2m+1}n+y_{m}).
\end{align*}

Hence, by \eqref{E:eqt31}, we conclude that
\begin{equation}\label{E:eqt32}
\Lambda(n,2m+1)=\max\big(g(2^{2m+1}n+y_{m}),g(2^{2m+1}n+x_{m})\big).
\end{equation}
This implies, also using Corollary \ref{c6}, that
\begin{align*}
\Lambda(2n,2m+1)&=\max\big(g(2^{2m+2}n+y_{m}),g(2^{2m+2}n+x_{m})\big)\\
&= g(2^{2m+2}n+y_{m}),\\
\Lambda(2n+1,2m+1)&=\max\big(g(2^{2m+2}n+2^{2m+1}+y_{m}),g(2^{2m+2}n+x_{m+1})\big)\\
&=g(2^{2m+2}n+x_{m+1}).
\end{align*}

And again, by \eqref{E:eqt31}, we find that
\begin{equation}\label{E:eqt33}
\Lambda(n,2m+2)=\max\big(g(2^{2m+2}n+y_{m}),g(2^{2m+2}n+x_{m+1})\big).
\end{equation}
The desired conclusion for $m$ follows from \eqref{E:eqt32} and \eqref{E:eqt33}. This achieves the proof by induction.
\end{proof}
\bg
\quad In particular, choosing $n=1$ and using Corollary \ref{c6}, we see that
\begin{align*}
\Lambda(1,2m-1)&=\max\big(g(2^{2m-1}+y_{m-1}),g(x_{m})\big)\\
&=g(x_{m})=g(y_{m})-v(y_{m}),\\
\intertext{and}
\Lambda(1,2m)&=\max\big(g(y_{m}),g(2^{2m}+x_{m})\big)\\
&=g(y_{m}).
\end{align*}
This is equivalent to
\begin{align*}
\Lambda(1,2m-1)&=\frac{6m-2+2^{1-2m}}{27}=\frac{1}{9}\left(2m-1+\frac{2^{2m-1}+1}{3}\cdot 2^{-(2m-1)}\right),\\
\Lambda(1,2m)&=\frac{6m+1-2^{-2m}}{27}=\frac{1}{9}\left(2m+\frac{2^{2m}-1}{3}\cdot 2^{-2m}\right),
\end{align*}
which can be expressed in a single formula as follows :
\begin{equation*}
\Lambda(1,m)=\frac{1}{9}\left(m+\frac{\round(2^m/3)}{2^m}\right)
\end{equation*}

\smallskip
So, we have proved the following two corollaries :

\smallskip
\begin{corollary}\label{c7}
For every nonnegative integer $m$ we have
\begin{equation*}
\max\big(g(t): 2^m\leq t<2^{m+1}\big)= \frac{1}{9}\left(m+\frac{\round(2^m/3)}{2^m}\right).
\end{equation*}
\end{corollary}

\smallskip
\begin{corollary}\label{c8}
 For every positive integer $n$ we have
\begin{equation*} 0\leq \frac{n(n+2)}{3}-G(n)\leq \theta_n
\leq \frac{1}{9}\floor{\lg n}+\frac{1}{18},
\end{equation*}
where 
$$\theta_n=\frac{1}{9}\left(\floor{\lg n}+\round(2^{\floor{\lg n}}/3)2^{-\floor{\lg n}}\right).$$
\end{corollary}

\medskip

\quad It is interesting to compare the upper bound of $g(n)$ given in Proposition \ref{p5}\eqref{itmp53}
 with the one given in Corollary \ref{c8} which is asymptotically the best possible by
Corollary \ref{c7}.

\smallskip

\quad Recall that the minimum of $g$ on $I_m$ is $0$ and that it is attained at a unique point $t_{\min}^{(m)}=2^{m+1}-1$. So, what about the maximum? By Corollary \ref{c7} the maximum of $g$ on $I_m$ is the number
$\lambda_m$ given by
\begin{equation}\label{E:eq50}
\lambda_m=\frac{1}{9}\left(m+\frac{\round(2^m/3)}{2^m}\right)=\frac{3m+1-(-1)^m2^{-m}}{27}.
\end{equation}

But what can one say about $s\in I_m$ knowing  $g(s)=\lambda_m$ ? The answer is 
in the following result.

\smallskip
\begin{proposition}\label{p10}
For every positive integer $m\geq2$. The function $g$ attains its 
maximum on $I_m$ at exactly two points 
$t_{\max}^{\prime(m)}$ and $t_{\max}^{\prime\prime(m)}$,  given by
\begin{equation*}
t_{\max}^{\prime(m)}=2^m+x_{\floor{m/2}},\quad\hbox{and}\quad
t_{\max}^{\prime\prime(m)}=2^m+y_{\floor{(m-1)/2}}.
\end{equation*}
\end{proposition}

\smallskip
\begin{proof}
 For $m=2$ the conclusion is clear, so let us suppose that $m\geq3$ and let us  consider an integer  $s$ satisfying
\begin{equation*}
2^m\leq s<2^{m+1}\quad\hbox{and} \quad g(s)=\lambda_m.
\end{equation*}
Clearly $s$ is even, since if $s=2s'+1$ for some $s'\in I_{m-1}$ then, using Proposition \ref{p5}\eqref{itmp51}, we have
$\lambda_m=g(s)=g(s')\leq\lambda_{m-1}$ which is absurd. So,  let us consider the following two cases :\bg
\begin{itemize}
\item $s\equiv2\mod4$. In this case we will prove that $s=t_{\max}^{\prime(m)}$. Indeed, suppose
that this is not true. It means that $s\equiv x_1\mod2^2$ and $s\not\equiv x_{k_0}\mod2^{2k_0}$ for $k_0=\floor{(m+1)/2}$, so
let us consider 
$$r=\max\{k\geq 1: s\equiv x_k\mod2^{2k}\}.$$
Clearly $r<\floor{(m+1)/2}$ (or equivalently $m\geq 2r+1$.) Moreover, by the definition of $r$ we have
$s=x_r+2^{2r}s'$ with $s'\not\equiv 2\mod4$ and $s'\in I_{m-2r}$. There are two cases :
\begin{itemize}
\item[$\scriptscriptstyle\square$] Either $s'=1+2p$ for some $p\in I_{m-2r-1}$,  which is absurd since, according to Corollary \ref{c6}, it leads to the following contradiction  :
$$\lambda_m=g(s)=g(2^{2r+1}p+2^{2r}+x_r)<g(2^{2r+1}p+y_r)\leq\lambda_m.$$
\item[$\scriptscriptstyle\square$] Or $s'=4p$ for some $p\in I_{m-2r-2}$, (this can happen only if $m\geq4$,) and this is also absurd since, according to Corollary \ref{c6}, it leads to the following contradiction :
$$\lambda_m=g(s)=g(2^{2r+2}p+x_r)<g(2^{2r+2}p+y_r)\leq\lambda_m.$$
\end{itemize}
\item[]  This proves that if $s\equiv2\mod4$ then $s=t_{\max}^{\prime(m)}$.\bg
\item   $s\equiv0\mod4$. By Proposition \ref{p6}, we have $g(\tilde{s})=g(s)=\lambda_m$ with $s\in I_m$ and
 $\tilde{s}\equiv2\mod4$. Therefore,
using the preceeding case we conclude that $\tilde{s}=t_{\max}^{\prime(m)}$ which is equivalent to
$s=\widetilde{t_{\max}^{\prime(m)}}=t_{\max}^{\prime\prime(m)}$.
\item[] This proves that if $s\equiv0\mod4$ then $s=t_{\max}^{\prime\prime(m)}$, and achieves
the proof of the proposition.
\end{itemize}

\end{proof}

\quad In fact, it is not difficult, by discussing according to the parity of $m$, to see that
\begin{equation*}
\left\{t_{\max}^{\prime(m)},t_{\max}^{\prime\prime(m)}\right\}=
\left\{2^m-1+\round\left(\frac{2^{m}}{3}\right),
2^m-1+\round\left(\frac{2^{m+1}}{3}\right),\right\}.
\end{equation*}
So, we have the following counterpart to Corollary \ref{c5}:

\smallskip
\begin{corollary}\label{c10}
We have 
\begin{align*}
\left\{n\geq1:G(n)=\frac{n(n+2)}{3}-\theta_n\right\}&=
\left\{2^r-1+\round\left(\frac{2^{r}}{3}\right):r\geq1\right\}\\
&\quad\bigcup
\left\{2^r-1+\round\left(\frac{2^{r+1}}{3}\right):r\geq1\right\}.
\end{align*}
\end{corollary}
\bg
{\bf Conclusion.} In this work, we studied certain sums related to the ``largest odd divisor'' function
and we obtained sharp bounds for these sums.
\bg
{\bf Problems.} Here we give some supplementary problems that can be easily solved with the
material presented in this article.

\medskip
\begin{description}
\item[{\bf Problem 1}] Let $\alpha$ be the function defined in \eqref{E:eq1}.
Prove that for $\beta>0$ we have, in the neighborhood of $+\infty$,
$$\sum_{k=1}^n\frac{\alpha(k)}{(n^2+k^2)^{1+\beta}}\sim \frac{1-2^{-\beta}}{3\beta}\cdot\frac
{1}{n^{2\beta}}.$$
What is the corresponding result when $\beta=0$ ?

\medskip
 \item[{\bf Problem 2}] (Japan Mathematical Olympiad 1993)  Let $U$ be the function defined in \eqref{E:eq3}.
Prove that there exists infinitely
many positive integers $n$ such that $3U(n)=2(1+2+\cdots+n)$. 

\medskip
\item[{\bf Problem 3}] Let $G$ be the function defined in \eqref{E:eq4}. Find all positive integers $n$ satisfying
$$ G(n)>\frac{n^2+2n}{3}-\frac{1}{4}.$$
\end{description}

\bg

\end{document}